\documentclass{amsart}
\usepackage{graphics}
\usepackage{graphicx}
 \usepackage[utf8]{inputenc}
 \usepackage[T1]{fontenc}
 \usepackage{lmodern}
 \usepackage[normalem]{ulem}
 \usepackage{verbatim}
 \usepackage{bbm}

 \usepackage{stmaryrd}

 \usepackage{amsmath}

 \usepackage{amssymb}
 \usepackage{dsfont}
\usepackage{amsthm}
\usepackage{pdfpages}
\usepackage{eso-pic}

\newtheorem{defi}{Definition}[section]
\newtheorem{theo}{Theorem}[section]%

\newtheorem{lemma}[theo]{Lemma}%
\newtheorem{rem}[theo]{Remark}%

\newcommand{\C}{\mathbb C}
\newcommand{\R}{\mathbb R}
\newcommand{\N}{\mathbb N}

\newcommand{\Ss}{\mathbb S}

\begin{document}
\title{Asymptotic Behavior of Multiplicative Spherical Integrals and S-transform }
\author{Jonathan Husson}
\address[Jonathan Husson]{Universit\'e de Lyon, ENSL, CNRS,  France}
\email{Jonathan.Husson@ens-lyon.fr}
\thanks{This work was supported in part by ERC Project LDRAM : ERC-2019-ADG Project 884584}

\maketitle

\begin{abstract}
	In this note, we study the asymptotics of a spherical integral that is a multiplicative counterpart to the well-known Harish-Chandra Itzykson Zuber integral. This counterpart can also be expressed in terms the Heckman-Opdam hypergeometric function. When the argument of this spherical integral is of finite support and of order $N$, these asymptotics involve a modified version of the $S$-transform of the limit measure of the matrix argument and its largest eigenvalue. To prove our main result, we are leveraging a technique of successive conditionning used in \cite{GuiHus}. In particular we prove in a "mathematically rigorous" manner a result from \cite{MerPot} in the case $\beta =1,2$ and we generalize it for multiple arguments. 
	\end{abstract}
\section{Introduction}

If $A$ and $B$ are two real or complex self-adjoint $N \times N$ matrices, the Harish-Chandra-Itzykson-Zuber integral is defined as: 
\[ HCIZ(A,B) := \int_{\mathcal{U}_N^{\beta}} \exp(N Tr(A U B U^*)) dU \]
where $dU$ is the Haar measure on either the orthogonal or the unitary group. Introduced by Harish-Chandra in \cite{Har}, and studied then by Itzykson and Zuber in \cite{ItzZub} who gave an explicit formula in the complex case, this integral has proven to be an interesting object in a wide variety of fields, from algebraic geometry \cite{HCIZalg,HCIZalg2} to physics \cite{Eyn}. In random matrix theory, it is a powerful tool when one looks at the large deviations of the spectrum. In \cite{GuiMai} Guionnet an Maida proved that when $A_N$ and $B_N$ are $N\times N$ matrices and $B_N$ is a matrix of rank one, the asymptotical behavior of $HCIZ(A_N,B_N)$ for large $N$ depends on the largest eigenvalue of $A_N$ and on the limit of its empirical measure through its $R$-transform. This result has since been extended by Guionnet and the author in \cite{GuiHus} to matrices $B_N$ with finite rank. It has been instrumental in the proofs of large deviations principles for the largest eigenvalue in various random matrix models \cite{HuGu,Hu,Giulio}. The HCIZ integrals are also related to Schur polynomials and the theory of special functions. In this paper, we study a multiplicative counterpart of the HCIZ integral that was comparatively less studied.  This integral is defined for a self-adjoint definite $N \times N$ positive matrix, $X$ and $\bar{\theta} = (\theta_1,...,\theta_N)\in \R^N$ by: 
\[ I_N(X,\bar{\theta}) = \int_{\mathcal{U}_N^{\beta}} \prod_{i=1}^N \det( [U^* X U]_i^{\theta_i - \theta_{i+1}} ) ^{\frac{\beta N}{2}}d U \]
where $[A]_i$ is the top left $i \times i$ block of $A$. The "multiplicative" denomination comes from the fact that if one takes two matrices $X$ and $Y$ and $U$ a Haar-distributed unitary matrix, one has: 

\[ \int_{\mathcal{U}_N^\beta} I_N(Z(U), \bar{\theta}) dU = I_N(X,\bar{\theta}) I_N(Y,\bar{\theta}).\]
where $Z(U)$ is a diagonal matrix with the same spectrum as $U^* X U Y$. When $\theta_i = \lambda_i$ with $\lambda$ a signature (that is a $N$-tuple of integers such that $\lambda_i \geq ... \geq\lambda_N$), this spherical integral can be expressed in function of the Jack polynomials $j_{\lambda}^{(\beta/2)}$. If we set $\textbf{x} = (x_1,...,x_N)$ the eigenvalues of $X$ then we have:

\[ I_N(X,\lambda) = \frac{ j_{\lambda}^{ ( \frac{\beta}{2} )}( \textbf{x})}{ j^{ ( \frac{\beta}{2} )}_{\lambda}( \textbf{1}_N)} \]
where $\textbf{1}_N= (1,...,1)$. 
In particular, when $\beta = 2$, we get back the Schur polynomials since $j^{(1)}_{\lambda}= s_{\lambda}$. For arguments that are not signatures, this spherical integral can also be expressed as a shifted version of the Heckman-Opdam hypergeometric function $\mathcal{F}^{(\beta)}_{\textbf{a}}(z)$ \cite{HeOp,OpBook}. Namely, we have that: 

\[  I_N(X,\bar{\theta}) = \mathcal{ F}^{(\beta)}_{ \ln \textbf{x}}(  \bar{\theta} - \rho) \]

where $ \ln \textbf{x} = ( \ln x_1,...,\ln x_N)$ and $\rho_i = \beta/2 (i -1 )$. The asymptotics of Schur and  Jack polynomials have been investigated in the literature, (see for instance \cite{2015}  for Schur polynomials and \cite{cuenca2017pieri,2018} for Jack polynomials and McDonald polynomials). However, our result deal with the logarithm of this integral divided by $N$ when the argument $\bar{\theta}$ has finite support whereas the preceding results of \cite{2015} deal with the case of $ \ln \textbf{x}$ having finite support. 
Another way to see this is to notice that we are interested in limits of the form:

\[   \lim_{ N \to \infty} \frac{1}{N} \ln  \mathcal{ F}^{(\beta)}_{\textbf{a}} \Big( \frac{\beta}{2}  ( \theta_1 N , \theta_2 N -1, ..., \theta_k N - k +1, -k,..., - N+ 1) \Big) .\]

Whereas previous results from \cite{2015,GuiHus} deal with limits of a different renormalization of the Heckman-Opdam function, that is: 
\[  \lim_{ N \to \infty} \frac{1}{N} \ln \frac{ \mathcal{ F}^{(\beta)}_{\textbf{a}}( \frac{\beta}{2}  ( \theta_1 N , \theta_2 N , .. \theta_k N , 0, 0,..., 0) )}{\mathcal{ F}^{(\beta)}_{\textbf{a}} (0, 0, 0,..., 0)  } .\]

   We direct the reader to \cite{MerPot} and the reference and the references therein for more information on the links between the quantity $I_N$ and special functions. In \cite{MerPot}, Mergny and Potters study the asymptotics of these integrals for a one-dimensional parameter $\bar{\theta}$ (that is a parameter whose support is at most on the first coordinate). This study reveals a modified version of the $S$-transform of the eigenvalue distribution and the largest eigenvalue. They also conjecture that for finite-dimensional parameters, the behavior of the log of this integral should be the sum of the one-dimensional behaviors. In this note we prove again the result of Potters and Mergny, (whose original proof make use of non-convergent contour integral) with mathematically rigorous tools and we also address their conjecture. It turns out that in this finite rank case, we have the same phenomenon as the one described for the HCIZ integral in \cite{GuiHus}. Asymptotically, the behavior of the log is indeed the sum of the one-dimensional behaviors but where we pair the $j$-th largest of the positive $\theta_i$ with the $j$-th largest eigenvalue of $X$ and the $j$-th smallest of the negative $\theta_i$ with the $j$-th smallest eigenvalue of $X$ (which correspond the the conjecture formulated in \cite{MerPot} if the extremal eigenvalues stick to the support of the limiting measure).

This proof follows the general philosophy of \cite{GuiHus} and we will do a recursion on the dimension of the argument $\bar{\theta}$. After the definition of the multiplicative spherical integral and the statement of the result in section \ref{def}, we will initialize our recursion and  prove the result in the one-dimensional case in section \ref{rank1}. For this we will proceed then in two steps: first, in subsection \ref{rank1dis} we will discretize the spectrum of the matrix $X_N$ in order to treat a simpler, finite-dimensional large deviation problem involving Dirichlet laws, and then in subsection \ref{rank1cont}, we will treat the generic case using discrete approximations and continuity arguments. Then, in section \ref{rank2}, we will assume that our theorem stands for the $k-1$-dimensional case and we will prove it for the $k$-dimensional one. There again we proceed first by discretization in subsection \ref{rank2dis} and then approximation and continuity in subsection \ref{rank2cont}. In subsection \ref{rank2dis}, we will use our recursion hypothesis by conditioning the first column of $U$ and reducing the problem to the study of another $k-1$-dimensional integral $I_{N-1}(\tilde{\theta},Y_N)$, where the spectrum of $Y_N$ is interlaced with the spectrum of $X_N$. Here, the main difference with \cite{GuiHus} is that we need to take into account the fact that contrary to the HCIZ integral, this integral is \emph{not} symmetrical in its argument $\theta$ which adds a bit of technicality in subsection \ref{rank2dis}. Furthermore the interlacing relationship between the spectra of $X_N$ and $Y_N$ will also be different.

\subsection{Notations}
 We will denote $\mathcal{H}_N^{\beta}$ the set of symmetric matrices of $\mathcal{M}_N(\R)$ if $\beta=1$ and the set of Hermitian matrices of $\mathcal{M}_N(\C)$ if $\beta =2$. We also denote $\mathcal{H}_N^{\beta,+}$ and $\mathcal{H}_N^{\beta,++}$ respectively the positive and definite positive matrices of $\mathcal{H}_N^{\beta}$. We will also denote $\mathcal{U}_N^{\beta}$ the orthogonal group if $\beta =1$ and the unitary group when $\beta =2$. For a matrix $X$ in $\mathcal{H}_N^{\beta}$ and $i \in \llbracket 1,N \rrbracket$, we will denote $\lambda_i(X)$ its $i$-th eigenvalue, $\chi_X(t) = \det(t - X)$ its characteristic polynomial and $\mu_{X} := N^{-1} \sum_{i=1}^N \delta_{\lambda_i(X)}$ its eigenvalue distribution. For $n \in \N^*$, we denote $\Ss^{\beta n -1}$ either the real sphere of dimension $n$ or the complex sphere of dimension $n$.  We will also denote $\R^+$ the set of non-negative real numbers and $\R^{+,*}$ the set $\R^+ \setminus \{0 \}$.

  If $\mu$ is a probability measure on $\R$, we will denote $r(\mu)$ the rightmost point of its support and $l(\mu)$ the leftmost point of its support. We denote $G_{\mu}$ its Stieljes transform, that is, for $z \in \C$ outside of the support of $\mu$: 
 
 \[ G_{\mu}(z) = \int_{t \in \R} \frac{1}{ z - t} d\mu(t) .\]
 
 Furthermore if $z = r(\mu)$ or $z = l(\mu)$, we define $G_{\mu}(z)$ as the following limits on the real line:
 
 \[ G_{\mu}( r(\mu)) = \lim_{z  \to r(\mu)^+} G_{\mu}(z) \text{ and }  G_{\mu}( l(\mu)) = \lim_{z  \to l(\mu)^-} G_{\mu}(z) .\]
  In particular, we allow $G_{\mu}(r(\mu)) = + \infty$ and $G_{\mu}(l(\mu)) = - \infty$. 

If, $\mu \in \mathcal{P}(\R^+)$, we also denote $T_{\mu}(z)= z G_{\mu}(z) - 1$. We define $T_{\mu}(r(\mu))$ and $T_{\mu}(l(\mu))$ as we did for $G_\mu$. We also denote, for  $\tilde{S}_{\mu}$ its modified $S$-transform defined on $[T_{\mu}(l(\mu)), T_{\mu}(r(\mu))] \setminus \{ 0, -1 \}$ by: 

\[ \tilde{S}_{\mu}(z) = \frac{z}{z + 1} T_{\mu}^{-1}(z) \]
where $T_{\mu}^{-1}$ is the inverse function of $T_{\mu}$ defined on $[T_{\mu}(l(\mu)), T_{\mu}(r(\mu))] \setminus \{ 0\}$. If $l(\mu) > 0$, doing a Taylor expansion of $T_{\mu}^{-1}$ in $-1$ show that we can extend $\tilde{S}_{\mu}$ analytically in $-1$ by $\tilde{S}_{\mu}(-1) = (\int d \mu(t)/t)^{-1}$ . We can also prove that $\tilde{S}$ can be analytically  extended on $0$ by $\tilde{S}_{\mu}(0)= \int t d \mu(t)$. Furthermore, derivating $\tilde{S}$ gives that it is an  increasing function on $] T_{\mu}(l(\mu)), T_{\mu}(r(\mu))[$.   
  
\section{Definitions and statement of the result}\label{def}
Given the preceding notations, we can define the multiplicative spherical integral on the group $\mathcal{U}^{\beta}_N$.
\begin{defi}
	For $\beta =1   \text{ or } 2$, $X_N \in \mathcal{H}_N^{\beta,++}$, $k \leq N$. $\overline{\theta} = (\theta_1, \theta_2,..., \theta_k)\in \R^k$ :
	\[ I_N(\overline{\theta}, X_N) = \int_{\mathcal{U}^{\beta}_N} \Delta_{\overline{\theta}}(U^* X_N U)^{\beta N/2} dU \]
	where $dU$ is the Haar measure on $\mathcal{U}_N^{\beta}$ and $\Delta_{a}(M)$ is defined for $M \in \mathcal{H}_N^{\beta,++}$ and $a \in \R^k$ by
	\[ \Delta_{a}(M) = \prod_{i=1}^{k} \det( [M]_i)^{a_i - a_{i+1}}  \]
	where $[M]_i$ is the $i \times i$ matrix obtained by extracting the first $i$ rows and the first $i$ columns of $M$ and with the convention $a_{k+1} =0$. 
	\end{defi}

 Here is our main result regarding the behavior of the multiplicative spherical integral. 
 \begin{theo}\label{maintheo} 
 	Let $\beta =1$ for the real case and $\beta =2$ for the complex case. Let $\overline{\theta}= (\theta_1,...,\theta_k)\in \R^k$, $\sigma$ a permutation of $ \llbracket 1,k \rrbracket$ and $m\in \llbracket 1,k \rrbracket$ such that: 
 	\[ \theta_{\sigma(1)} \leq \theta_{\sigma(2)}\leq ... \leq \theta_{\sigma(m)} \leq 0 \leq \theta_{\sigma(m+1)} \leq ... \leq \theta_{\sigma(k)} .\]
 	Let $(X_N)_{N \in \N}$ be a sequence of matrices in $\mathcal{H}^{\beta,++}_N$ such that there exists $M \in \R^+$ so that $||X_N||, ||X_N^{-1}||\leq M$ for all $N$ and such that there exists $\lambda_1,... \lambda_k \in \R^{+,*}$ such that:
 		 \[ \forall i \in [1,m], \lim_{N \to \infty} \lambda_i(X_N) = \lambda_i \]
 		 and \[ \forall i \in [m+1,k], \lim_{N \to \infty} \lambda_{N +i -k}(X_N) = \lambda_i \]
 		 then: 
 		 \[\lim_{N \to \infty} \frac{1}{N} \log I_N(\bar{\theta},X_N) =\frac{\beta}{2} \sum_{i=1}^k J(\theta_{\sigma(i)}, \lambda_i, \mu) \]
 	where $J$ is defined for $(\theta, \lambda, \mu)$ in the following set :
 	\[ \{ (\theta, \lambda, \mu) \in   \R^+ \times \R^{+,*} \times \mathcal{P}(\R^+) :  r(\mu) \leq \lambda \} \cup \{ (\theta, \lambda, \mu) \in   \R^- \times \R^{+,*} \times \mathcal{P}(\R^+) :  l(\mu) \geq \lambda, r(\mu) < + \infty \} \]
 	by 
 	\[ J(\theta, \lambda, \mu)
 	:= (\theta+1) \log c - \log |\theta| - \int \log \Big| \frac{\theta+1}{\theta} c - x \Big| d \mu(x) \]
 	 for $\theta \neq 0, -1$ with: 
 	\[ c := \begin{cases} \frac{\theta}{\theta +1} \lambda \text{ when } \theta \leq T_{\mu}(\lambda) \leq 0 \text{ and } 0 \leq T_{\mu}(\lambda) \leq \theta \\
 	\tilde{S}_{\mu}(\theta) \text{ otherwise}. \end{cases}
 	\]	
 	For $\theta = -1$, we can extend $J$ by continuity in $\theta$ by letting:
 	
 	\[ J(-1, \lambda, \mu) = - \int \log |x| d \mu(x). \]
 	For $\theta =0$, we can extend $J$ by continuity by letting : 
 	\[ J(0, \lambda, \mu) = 0 .\] 
 \end{theo}	
One can notice in particular that for values of $\theta$ between $0$ and $T_{\mu}(\lambda)$, $J(\theta, \lambda, \mu)$ does not depend on $\lambda$. In particular, if we differentiate in $\theta$, we can see that 
\[ \frac{\partial J}{\partial \theta}(\theta, \lambda, \mu) = \log \tilde{S}_{\mu}(\theta) .\]
Therefore: 

\begin{rem}
	If $\theta \in [0, T_{\mu}(\lambda)]$ for $\theta > 0$ or if $\theta \in [ T_{\mu}(\lambda),0]$ for $\theta < 0$, we have: 
	\[ J(\theta, \lambda, \mu) = \int_0^\theta \log \tilde{S}_{\mu}(t)d t. \]
	\end{rem}
This expression is similar to the expression found in \cite[Theorem 6]{GuiMai} which expresses the limit of classical spherical integrals using the integral of the $R$-transform. Before further studying this integral, one can make a first useful remark about its regularity with respect to the parameters $\overline{\theta}$ and the eigenvalues of $X_N$.

\begin{rem}\label{equicont}
 	\begin{itemize}
 		Let $\overline{\theta}, \overline{\phi} \in \R^k$ and $X_N,Y_N \in \mathcal{H}_N^{\beta,++}$ two definite positive matrices. 
 		\item If there is $\epsilon >0$ such that for all $i \in [1,N]$, 
 		\[ e^{- \epsilon} \leq \frac{\lambda_i(X_N)}{\lambda_i(Y_N)} \leq e^{\epsilon} \]
 		then 
 		\[ \left| I_N(X_N,\overline{\theta}) - I_N(Y_N,\overline{\theta}) \right| \leq M \epsilon \]
 		where $M = k \max_{i=1}^{k-1} |\theta_{i+1} - \theta_i|.$ 
 		\item if for all $i \in [1,k]$
 		\[|\theta_i - \phi_i| \leq \epsilon \]
 		then 
 		\[ \left| I_N(X_N,\overline{\phi}) - I_N(X_N,\overline{\theta}) \right| \leq M' \epsilon \]
 		where $M' = 2k \max_{i=1}^N| \log \lambda_i(X_N)|. $
 	\end{itemize}	
 \end{rem}

\begin{proof}
	For the first point, let us compare $\det([X_N]_i)$ and $\det([Y_N]_i)$ for $i \leq k$. There exists an orthonormal basis $e_1,...,e_i \in \C^i$ such that:
	\[ \det([X_N]_i)= \prod_{j=1}^i \langle e_j, [X_N]_i e_j \rangle  .\]
	But the hypothesis implies that for every $u \in \C^N$, $\langle u, X_N u\rangle \geq e^{- \epsilon} \langle u, Y_N u\rangle$ and this imply that for every $u \in \C^i$, $\langle u, [X_N]_i u \rangle \geq e^{- \epsilon} \langle u, [Y_N]_i u \rangle$. Therefore, we can write: 
	
	\[ \det([X_N]_i) \geq e^{-i  \epsilon}  \prod_{j=1}^i \langle e_j, [Y_N]_i e_j \rangle \geq  e^{- i \epsilon} \det( [Y_N]_i) .\]
	Where we used Hadamard's inequality for the last inequality. Switching the roles of $X_N$ and $Y_N$, we also have that $\det( [X_N]_i) \leq \exp(i \epsilon) \det([Y_N]_i)$. It follows then easily that $\exp( M \epsilon) \Delta_{\overline{\theta}}(Y_N) \geq \Delta_{\overline{\theta}}(X_N) \geq \exp(- M \epsilon) \Delta_{\overline{\theta}}(Y_N)$ which implies our result.

	For the second point, it suffices to remark that:
	
	\[ \exp( - M' /2) \leq  \det[X_N]_i \leq \exp( M '/2) .\]
	\end{proof}
\section{Proof in the case $k =1$}\label{rank1}
 
 We will begin by proving the Theorem \ref{maintheo} in the case $k=1$ assuming that the spectra of the $X_N$ are all included in some finite set $\{ \mu_1, ..., \mu_l \}$ in the subsection \ref{rank1dis}. Then, in the subsection \ref{rank1cont}, we will generalize this to general matrices using approximations and continuity arguments for the integral $I_N$ and its limit $J$. This case will be refered as the one-dimensional case.
 \subsection{The case of spectra included in a finite set}\label{rank1dis}
 
Using the unitary invariance of the Haar measure, without loss of generality, we can assume that $X_N$ is diagonal.
In this subsection, we assume that there exists $0 < \mu_1 ... < \mu_l$ and integers  $\alpha_i(N)$ such that $\alpha_1(N)+...+\alpha_l(N) = N$ and $\lim_{N \to \infty} \alpha_i(N)/ N = \alpha_i \geq 0$ so that $X_N$ is of the form  
\[ X_N= diag( \underbrace{\mu_1,...,\mu_1}_{\alpha_1(N) \text{ times }},...,\underbrace{\mu_l,...,\mu_l}_{\alpha_l(N) \text{ times }} ).\]
There, we have the convergence of the empirical $\mu_{X_N}$ toward $\mu = \sum_{i=1}^l \alpha_i \delta_{\mu_i}$. Furthermore, let $\theta \in \R$. If $\theta \geq 0$ we assume that $\alpha_l(N) > 0$ for $N$ large enough and if $\theta <0$, we assume that $\alpha_1(N) > 0$. Then we want to prove the following lemma which corresponds to Theorem \ref{maintheo} for $k =1$ 
\begin{lemma}\label{lemrank1}
	If $\theta \geq 0$, then: 
	\[ \lim_{N \to \infty} \frac{1}{N} \log  I_N( \theta, X_N) = \frac{\beta}{2} J(\theta, \mu_l, \mu) .\] 

If $\theta \leq 0$, then: 
\[ \lim_{N \to \infty} \frac{1}{N} \log  I_N( \theta, X_N) = \frac{\beta}{2}J(\theta, \mu_1, \mu). \] 
\end{lemma}
\begin{proof}

 We will denote $I_i^N$ the following intervals: 
\[ I_i^N = \rrbracket \sum_{j=1}^{i-1} \alpha_j(N), \sum_{j=1}^i \alpha_j(N) \rrbracket . \]
If we denote $e$ the first column of $U$, we have that 
\[ \det([U^* X_N U]_{1})^{\frac{\beta\theta N}{2}} = \langle e, X_N e \rangle^{\frac{\beta\theta N}{2}} = \Big( \sum_{j=1}^l \mu_i \gamma_i \Big)^{\frac{\beta N \theta}{2}}, \]
where the $\gamma^N_i$ are defined as follows: 
\[ \gamma^N_i = \sum_{j \in I^N_i} |e_j|^2. \]
$(\gamma^N_1,...,\gamma^N_l)$ is a random variable whose distribution is a Dirichlet law of parameters $( \beta \alpha_1(N)/2,..., \beta \alpha_l(N) /2 )$. Therefore, following for instance \cite[Theorem 11]{GuiHus} it satisfies a large deviations principle on the $l-1$-dimensional simplex $K := \{ \gamma \in (\R^{+})^l : \gamma_1 + ... + \gamma_l = 1\}$ with good rate function: 
\[ \frac{\beta}{2}H(\gamma) = - \frac{\beta}{2}\sum_{i=1}^l \alpha_i \log \frac{\gamma_i}{\alpha_i} .\]

Therefore, using Varadhan's lemma we have that:

\[ \lim_{N \to \infty} \frac{1}{N} \log I_N(\theta, X_N) = \frac{\beta}{2} \sup_{\gamma \in K} f_{\theta}(\gamma) \text{ where} f_{\theta}(\gamma) = \Big[\theta \log  \Big( \sum_{i=1}^l \gamma_i \mu_i \Big) - H( \gamma) \Big]. \]

Since the function $f_{\theta}$ is proper, and continuous from $K$ to $\R \cup \{ - \infty \}$ and since it is upper bounded by $|\theta \max \{ \log \mu_1, \log \mu_l \}|$, its supremum is necessarily achieved at some $\overline{\gamma}$. Since a priori, there may be several maximum arguments, we choose one $\overline{\gamma}$ arbitrarily. However we will end up proving that there is indeed only one maximal argument possible.

Let us first deal with the case of a positive $\theta$. Let $A$ be the set of indices $i$ such that $\alpha_i =0$. Then we can assume $\bar{\gamma}_i = 0$ for all $ i \in A \setminus \{ l\}$. Indeed for any $\gamma \in K$, if we construct $\psi \in K$ by letting $\psi_l = \sum_{i\in A \setminus \{l \}} \gamma_i + \gamma_l $, $\psi_i = 0$ for $i \in  A \setminus \{l\}$ and $\psi_i = \gamma_i$ for all other indices, we have easily that $f_{\theta}(\psi) \geq f_{\theta}(\gamma)$. Therefore, up to reasoning on $K \cap \R^{ \llbracket 1, l \rrbracket \setminus ( A \setminus \{ l \})}$ rather that on $K$, we can assume that $\alpha_i >0$ for all $i$ except maybe $l$. Then, there are two possibilities: either $\bar{\gamma}_l =0$ or $\bar{\gamma}_l >0$. If $\nabla f_{\theta}$ is the gradient of $f_\theta$ in the ambient space $\R^l$, we have that: 

\[ \forall i =1,...,l, \, \nabla f_{\theta}(\gamma) _i = \frac{\theta \mu_i }{\sum_{j=1}^l \gamma_j \mu_j } + \frac{\alpha_i}{\gamma_i} .\]

Then, writing the critical point equation for each cases, we have that there exists $a \in \R$ such that respectively: 

\begin{align*}
 \text{ if } \bar{\gamma}_k = 0, & \, \forall i =1,...,l-1, \, \nabla f_{\theta}(\bar{\gamma})_i = a,  \\\text{ if } \bar{\gamma}_k > 0, &\, \forall i =1,...,l, \, \nabla f_{\theta}(\bar{\gamma})_i = a.  \end{align*}

In both cases, it implies $\sum_{i=1}^l  \bar{\gamma}_i \nabla f_{\theta}(\bar{\gamma})_i = a $ and we get that $a = \theta +1$. If we let $c = \sum_{j=1}^l  \mu_j \bar{\gamma}_j$ we have that for every $i =1,...,l-1$:
\[ \frac{\theta \mu_i}{c} + \frac{\alpha_i}{\bar{\gamma}_i} = \theta + 1 .\]
We remind that if $\alpha_i > 0$ then $\bar{\gamma}_i >0$. Indeed, if not we would have $f_{\theta}(\bar{\gamma}) = - \infty$. Let us assume that $\bar{\gamma}_l >0$ and $\alpha_l =0$ then:
\[ c = \frac{\theta \mu_l}{\theta +1} \]
and by definition of $c$, we must have: 
\begin{eqnarray*} c &=& \sum_{i=1}^{l} \mu_i \bar{\gamma}_i \\
	& = &\sum_{i =1}^{l-1} \frac{\mu_i \alpha_i}{\theta +1 - \theta \frac{\mu_i}{c}} + \mu_l \bar{\gamma}_l \\
	& > & \frac{1}{\theta +1}\sum_{i =1}^{l-1} \frac{\mu_i \alpha_i}{1 -  \frac{\mu_i}{\mu_l}} \\
	& > & \frac{c}{\theta} \sum_{i=1}^{l-1} \frac{\mu_i \alpha_i}{\mu_l - \mu_i} \\
	& >& \frac{c}{\theta} T_{\mu}(\mu_l).
	\end{eqnarray*}
Therefore, we necessarily have $\theta > T_{\mu}(\mu_l)$. If now we assume that either $\bar{\gamma}_l =0$ (and therefore $\alpha_l = 0$) or $\alpha_l > 0$ (and therefore $\bar{\gamma}_l >0$), then we have
\begin{eqnarray*} c &=& \sum_{i=1}^l  \mu_i \bar{\gamma}_i \\
	& = &\sum_{i =1}^{l} \frac{\mu_i \alpha_i}{\theta +1 - \theta \frac{\mu_i}{c}} 
\end{eqnarray*}
so that 
\[ \theta = T_{\mu}(\frac{\theta +1}{\theta} c ) .\]
Furthermore, since $\alpha_{i} > 0$, for $i =1,...,l-1$, we necessarily have $\overline{\gamma}_i > 0$ for $i =1,...,l-1$. Therefore, if $u = (-(l-1)^{-1},...,-(l-1)^{-1},1))$, $\overline{\gamma} + \epsilon u \in K$ for $\epsilon >0$ small enough and so, differentiating in $\epsilon \to 0^+$, we must have that: 
\[ \langle \nabla f_{\theta} ( \bar{\gamma}) , u \rangle \leq 0 \]
	which implies that 
	\[ \theta \frac{\mu_l}{c} - (\theta +1) \leq 0 \]
	that is $\mu_l \leq \frac{\theta +1}{\theta}c$  so that, applying $T_{\mu}$ we have $\theta \leq T_{\mu}( \mu_l)$. 
	Therefore, since at least one maximal argument $\bar{\gamma}$ always exists, we have the following dichotomy for every possible $\bar{\gamma}$: 
\begin{itemize}
	\item If $\theta \leq T_{\mu}( \mu_l)$, then either $\bar{\gamma}_l =0$ or $\alpha_l  >0$ and in both cases $c = \frac{\theta}{\theta +1} T_{\mu}^{(-1)}( \theta) = \tilde{S}_{\mu}(\theta)$.
	
	\item If $\theta > T_{\mu}( \mu_l)$, then $\bar{\gamma}_l > 0$ and $c = \frac{\theta}{\theta +1} \mu_l$.
\end{itemize}
In each case, we remind that we have for $i = 1,...,l-1$: 
\[ \bar{\gamma}_i = \frac{\mu_i \alpha_i}{\theta+1 - \theta \frac{\mu_i}{c}} = \frac{c}{\theta} \frac{\alpha_i \mu_i}{\frac{\theta +1}{\theta} c - \mu_i}.\]	
and  
\[\bar{\gamma}_l = 1 - \sum_{i=1}^{l-1} \bar{\gamma}_i .\]	
	
Therefore, there can be only one possible $\bar{\gamma}$ and then substituting $\bar{\gamma}$ and $c$ in $f_{\theta}(\gamma)$ we have:
\begin{eqnarray*} \lim_{N \to \infty} \frac{1}{N} \log I_N(\theta, X_N) &=& \frac{\beta}{2} \Big[(\theta +1) \log c - \log |\theta| - \sum_{i=1}^{l-1} \alpha_i \log \Big| \frac{\theta +1}{\theta}c - \mu_i \Big|\Big]\\
	&= & \frac{\beta}{2} J(\theta, \lambda,\mu)  \end{eqnarray*}
 where $J$ is indeed the function defined in Theorem \ref{maintheo}.

 For $\theta \leq 0$, we have the same analysis by replacing $\mu_l$ by $\mu_1$ (and thus assuming that $\alpha_i > 0$ for all $i >1$). We end up then with the following values for $c$: 
 \begin{itemize}
 	\item If $\theta \geq T_{\mu}( \mu_1)$, then $c = \frac{\theta}{\theta +1} T_{\mu}^{(-1)}( \theta) = \tilde{S}_{\mu}(\theta)$ (with the continuous extension in $\theta = -1$, $\tilde{S}_{\mu}(-1)= \Big( \int \frac{d\mu(x)}{x} \Big)^{-1}$).
 	
 	\item If $\theta < T_{\mu}( \mu_1)$, then $c = \frac{\theta}{\theta +1} \mu_1$.
 	Then again, substituting $\hat{\gamma}$ and $c$, we obtain again the function $J$.  
 	\end{itemize}
 \end{proof}
 \subsection{Generalization to a generic limit $\mu$}\label{rank1cont}
 Now we want to generalize our result to matrices with a general limit measure $\mu$. First we will need that the function $J$ is continuous in the argument $(\theta,\lambda_,\mu)$. 
 
 \begin{lemma} \label{cont}
 	Let $\kappa > 0$. The function $J$ is continuous on the sets: 
 	\[\mathcal{D}^+ : = \{ (\theta,\lambda,\mu) \in \R^+ \times \R^{+,*} \times \mathcal{P}( \R^+), \lambda \geq r(\mu), l(\mu) > 0 \} \]
 	
 	and on the set 
 	\[\mathcal{D}^- : = \{ (\theta,\lambda,\mu) \in \R^- \times \R^{+,*} \times \mathcal{P}( \R^+), \lambda \leq l(\mu), r(\mu) < \kappa \} \]

 	where $\mathcal{D}$ is endowed with the topology induced by the product topology where we consider the weak topology on $\mathcal{P}(\R^+)$.
 	\end{lemma}
\begin{proof}
Let $ d = \frac{\theta +1}{\theta} c$ so that 
\[ J(\theta, \lambda, \mu) = (\theta +1) \ln d +  \theta \ln |\theta| - (\theta+1) \ln (|\theta +1|) - \int \ln | d - x | d \mu(x) .\]
Let us first prove the continuity on $\mathcal{D}^+$. Differentiating in $\lambda$, we get that 
\[ \frac{\partial J }{\partial \lambda} = \begin{cases} \frac{\theta + 1}{\lambda} - G_{\mu}(\lambda) \text{ if } \lambda \geq T_{\mu}^{-1}( \theta) \\
0 \text{ if not } \end{cases} .\]
This implies that since $G_\mu (\lambda) \geq \frac{1}{\lambda}$. 

\[ 0 \leq \frac{\partial J }{\partial \lambda}(\theta, \lambda, \mu) \leq \frac{\theta}{\lambda} .\]

so that on $\mathcal{D}^+$:

\[ |J(\theta,\lambda,\mu) - J(\theta,\lambda',\mu)| \leq \theta | \ln \lambda - \ln \lambda' |. \]

The same way, differentiating in $\theta$ we have:
\[ \frac{\partial J}{\partial \theta}(\theta, \lambda, \mu) = \ln c. \]
For $\theta \leq T_{\mu}(\lambda)$, since $\tilde{S}$ is increasing, we have that $ c = \tilde{S}(\theta) \leq \lambda$. If not, since $\theta > 0$, $c = \frac{\theta}{\theta +1 } \lambda \leq \lambda$. Therefore: 
\[ 0 \leq \frac{\partial J}{\partial \theta} ( \theta, \lambda,\mu) \leq \ln \lambda .\] 

And then we use that $c \leq \lambda$ to have that 

\[ |J(\theta, \lambda, \mu) - J(\theta', \lambda, \mu)| \leq \ln \lambda |\theta - \theta'| .\]

Therefore, we have a continuity modulus uniform on every compact on $\mathcal{D}^{+}$ for the variables $(\theta, \lambda)$ in the sense that for each such compact $K \subset \mathcal{D}^+$, the exists some increasing function $g_{K} : \R^{+} \to \R^+$ that tends to $0$ in $0$ such that if $(\theta, \lambda,\mu),(\theta', \lambda',\mu) \in K$ such that $| \lambda - \lambda'|+ |\theta - \theta'| \leq \epsilon$ then 
\[ |J((\theta, \lambda,\mu)) - J((\theta', \lambda',\mu))| \leq g_K(\epsilon) .\]   `Therefore it suffices to prove continuity with respect to $\mu$. Using the continuity modulus with respect to $\lambda$, it suffices to prove the continuity on the domains:

	\[\mathcal{D}_{M,A, \epsilon}^+ : = \{ (\theta,\lambda,\mu) \in [0, A] \times [M^{-1}, M ] \times \mathcal{P}( \R^+),  \lambda \geq e^{\epsilon} r(\mu), 0 \leq l(\mu) \} \]
	for $M >0, \epsilon > 0, A >0$. If $((\theta_n,\lambda_n, \mu_n))_{n \in \N}$ is a sequence of elements of $\mathcal{D}_{M,A, \epsilon}^+$ converging toward converging toward some $(\theta, \lambda, \mu) \in \mathcal{D}_{M,A, \epsilon}^+$, we see that for $n$ large enough $(\theta, \lambda, \mu_n) \in \mathcal{D}_{M+1,A+1, \epsilon/2}^+ $ and using the uniform continuity modulus on $\mathcal{D}_{M+1,A+1, \epsilon/2}^+$, we see that it is sufficient to show that $J(\theta, \lambda, \mu_n)$ converges toward $J(\theta, \lambda, \mu)$.

	 If $\theta = 0$, the result is trivial, so we can assume that $\theta > 0$. Then let $d_n = \frac{\theta +1}{\theta}  c_n$, where $c_n$ is defined the same as the quantity $c$ for $\lambda_, \theta, \mu_n$ instead of $\lambda, \theta, \mu$. First we have $\lim_{n \to \infty} d_n =d$. Indeed, if $\theta > T_{\mu}(\lambda)$, using the fact that $\lambda \geq e^{\epsilon/2} r(\mu_n)$ for $n$ large enough  and the weak convergence of $\mu_n$, we have that $\lim_{n \to \infty} T_{\mu_n}(\lambda) = T_{\mu}(\lambda)$, and therefore $\theta > T_{\mu_n}(\lambda)$ and  $d_n = \lambda$ for $n$ large enough. If $\theta < T_{\mu}(\lambda)$, the same argument leads to $\theta <  T_{\mu_n}(\lambda)$ for $n$ large enough. In that case, for every $\eta >0$, we can find $\theta^+$ and $\theta^-$ such that $\theta^- < \theta < \theta^+$ and $T_{\mu}^{-1}(\theta^-) \geq T_{\mu}^{-1}(\theta) - \eta> e^{\epsilon/2} r(\mu)$ and $T_{\mu}^{-1}(\theta^+) \leq T_{\mu}^{-1}(\theta) + \eta$. If we denote $d^{\pm} = T_{\mu}^{-1}(\theta^{\pm})$, using the weak convergence of $\mu_n$ and the fact that $d^- \geq e^{\epsilon/2} r(\mu_n)$ for $n$ large enough, we have that $T_{\mu_n}(d^- ) \geq \theta \geq T_{\mu_n}(d^+) $ and therefore $d^- \leq d_n \leq d^+ $ for $n$ large enough which proves $\lim_{n \to \infty} d_n = d$. The case $T_{\mu}(\lambda)  = \theta$ can be easily deduced by uniform continuity in $\theta$. To conclude then, we only need to prove the convergence of the term $\int \log | d_n - x| d \mu_{n}$. One can see that $d_n \geq \lambda > e^{\epsilon/2} r(\mu_n)$. Therefore we have $d_n \geq e^{\epsilon/2} r(\mu)$. The sequence of function $x \mapsto \log | d_n -x |$ converges uniformly on $]0, e^{\epsilon/3} r(\mu)]$ and one concludes with the weak convergence of $\mu_n$. The proof is the same for $\mathcal{D}^-$ except one needs the assumption $r(\mu_n)\leq \kappa$ to prove the uniform convergence of the $x \mapsto \log |d_n - x|$ on some $[e^{- \epsilon/3} l(\mu), \kappa]$ and that one needs also to consider the case $\theta = -1$ apart from the generic case.

 \end{proof}

To proceed from the case with $\mu$ with a finite support to the case of a diffuse support, we first consider the case with $\mu$ with a continuous repartition function. We will assume that $\theta \geq 0$ (the case $\theta \leq 0$ is similar). Let us take $(X_N)_{N\in \N}$ a sequence of matrices such that their eigenvalue distributions converge toward $\mu$, that $\lambda_N(X_N)$ converges toward some $\lambda_1$ and that $||X_N||, ||X_N^{-1}||\leq M$ for some $M$. For $\epsilon > 0$ and let $X_N^{(\epsilon)}$ be the modification of $X_N$ where we replace its $i$-th eigenvalue $\lambda_i(X_N)$ by  $\lambda_i^{(\epsilon)}(X_N) = (1+ \epsilon)^{n_0}$ where $n_0$ is the smallest integer such that $\lambda_i(X_N) \leq (1+ \epsilon)^{n_0}$. Then the eigenvalue distribution of $X_N^{\epsilon}$ converges toward $\mu^{(\epsilon)} = \sum_{k\in \N^*} \alpha_k \delta_{(1+ \epsilon)^k}$ where $\alpha_k = \mu( [(1+ \epsilon)^{k-1}, (1+ \epsilon)^{k}])$ thanks to the continuity hypothesis on the repartition function of $\mu$. Since $\lambda_N(X_N)$ converges toward $\lambda_1$,  $\lambda_N(X_N^{(\epsilon)})$ is constant for $N$ large enough and equal to $\lambda_1^{(\epsilon)}$ (defined the same way as $\lambda_i(X_N^{(\epsilon)})$). Then $X_N^{(\epsilon)}$ satisfies the conditions of section \ref{rank1dis} and Lemma \ref{lemrank1} applies. Therefore: 

	\[ \lim_{N \to \infty} \frac{1}{N} \log I_N(X^{(\epsilon)}_N, \theta) = \frac{\beta}{2}J(\theta, \lambda_1^{(\epsilon)}, \mu^{(\epsilon)}). \]

With our hypothesis on $\mu$, $\lim_{\epsilon \to 0}\mu^{(\epsilon)} = \mu$ and $\lim_{\epsilon \to 0} \lambda_1^{(\epsilon)} = \lambda_1$. Furthermore, using Remark \ref{equicont}, we can see that:
\[ \frac{1}{N} | \log I_N(\theta,X_N) -  \log I_N(\theta,X^{(\epsilon)}_N) | \leq \frac{\beta}{2} \theta \log(1 + \epsilon) .\]
Using the continuity of the limit $J$, one concludes Theorem \ref{maintheo} stands in this case.

To relax the condition on the repartition function of  $\mu$ we can do the same thing by replacing $\lambda_i(X_N)$ by $\lambda_i ^{(\epsilon)}(X_N)= \lambda_i(X_N)(1 + \epsilon \delta_i)$ where $\delta_i \in [0,1]$ and are sampled so that the eigenvalue distribution of $X_N^{(\epsilon)}$ converges towards $\mu^{(\epsilon)}$ which is the independent product of $\mu$ with a uniform law on the segment $[1, 1 + \epsilon]$. This law has a continuous repartition function and we can again apply the approximation arguments above. 

\section{The case of $k \geq 2$}\label{rank2}

To prove the result in the case $k \geq 2$, we will proceed recursively on $k$ by conditioning on the first column of $U$. This will bring us back the case $k-1$. As in the one-dimensional case, we will first proceed by dealing with the case where the eigenvalues of $X_N$ are in a finite set in subsection \ref{rank2dis} and then we will extend the result to general $\mu$ in the same fashion as in the one-dimensional case in subsection \ref{rank2cont}.  

Therefore, for the rest of this section, we let $k > 1$ be an integer and we assume that the Theorem \ref{maintheo} is satisfied for $k-1$. 

\subsection{Spectra included in a finite set}\label{rank2dis}

Let $\bar{\theta} \in \R^{l}$,  $\sigma$ a permutation of $[2, k-1]$ and $m \in [1, k-1]$ such that $\theta_{\sigma(1)} \leq \theta_{\sigma(2)}\leq ... \theta_{\sigma(m)} \leq 0 \leq  \theta_{\sigma(m+1)} \leq ... \leq \theta_{\sigma(k-1)}$. Let $p \in \N$ such that $\theta_{\sigma(p)} \leq \theta_1 \leq \theta_{\sigma(p+1)}$. We will first assume that $\theta_1 \geq 0$ which implies that we can choose $p \geq m$. At the end of the proof, we will discuss the differences that arise when $\theta_1 \leq 0$

We will assume that there is a finite number of reals $\mu_{-m} < \mu_{- m+1} < ... < \mu_{l + k-m}$ such that:
\[ X_N= diag(\mu_{-m},...,\mu_{-1},  \underbrace{\mu_0,...,\mu_0}_{\alpha_0(N) \text{ times }},...,\underbrace{\mu_l,...,\mu_l}_{\alpha_l(N) \text{ times }}, \mu_{l+1},...,\mu_{l+k-m} ) \]
 and where $\lim \alpha_i(N)/N  = \alpha_i \geq 0$. 

Therefore, given these notation we are trying to establish is the following lemma:
\begin{lemma}

\begin{align*}
\lim_{N \to \infty} \frac{1}{N} \log I_N(\overline{\theta},X_N)& =  \frac{\beta}{2}J(\theta_1, \mu_{l+ p-m+1}, \mu)
+ \frac{\beta}{2}\sum_{i=1}^{m} J(\theta_{\sigma(i)}, \mu_{i-m-1}, \mu) \\
&+ \frac{\beta}{2}\sum_{i= m+1}^{p } J(\theta_{\sigma(i)},\mu_{l+i-m},\mu) + \frac{\beta}{2} \sum_{i=p+1}^{k-1} J(\theta_{\sigma(i)}, \mu_{i - m +l+1}, \mu) \end{align*}

\end{lemma}
\begin{proof}
First, we can write the matrix $U^* X_N U$ by blocks as follows 
\[ U^* X_N U = \begin{pmatrix} a & c^* \\ c & \tilde{X}_{N -1} \end{pmatrix} \]

where $a \in \R^{+,*}$, $c \in \C^{N-1}$ and $\tilde{X}_{N -1} \in \mathcal{H}^{\beta, ++}_{N-1}$. Then one can see that: 

\[ T^*_{N} U^* X_N U T_N = \begin{pmatrix} a & 0 \\ 0 & \tilde{X}_{N -1} - \frac{c c^*}{a} \end{pmatrix} \]

where 
\[ T_N = \begin{pmatrix} 1 & - \frac{c^*}{a} \\ 0 & I_{N-1} \end{pmatrix} .\]
 Due to the the fact that $T_N$ is upper triangular, this implies that for $i \in [1,N]$: 
 
 \[ [ T^*_N ]_{i} [U^* X_N U]_{i}[ T_N ]_{i} = \begin{pmatrix} a & 0 \\
 0 & [ \tilde{X}_{N -1} - \frac{c c^*}{a}]_{i-1} \end{pmatrix}\]
 
 If we denote $Y_N= \tilde{X}_{N -1} - \frac{c c^*}{a}$, using that $\det[T_N]_i =1$, we have that $\det([U^* X_N U]_{i}) = a \det(Y_N)$ and we can express $\Delta_{\bar{\theta}}(X_N)$ as a function of $a$ and $Y_N$.
 \[ \Delta_{\bar{\theta}}(X_N) = a^{\theta_1 - \theta_2} \prod_{i=2}^k \det[ U^* X_N U ]_{i}^{\theta_{i} - \theta_{i-1}} = a^{\theta_1 - \theta_2} \prod_{i=2}^k a^{\theta_{i} - \theta_{i+1}} \det[ Y_N]_{i}^{\theta_{i} - \theta_{i+1}} = a^{\theta_1} \Delta_{\tilde{\theta}}(Y_N) \]
 where $\tilde{\theta} = ( \theta_2, \theta_3, ....,\theta_k) \in \R^{k -1}$. Furthermore, one can easily see that the spectrum of $Y_N$ only depends on the first column of $U$ and that once conditioned on this that column  that we denote $e$, the law of $Y_N$ is invariant by conjugation by any element of $\mathcal{U}^{\beta}_{N-1}$. Therefore, since $a = \langle e, X_N e \rangle$ we deduce that:
 
 \begin{eqnarray}I_N(\bar{\theta}, X_N) & =& \int_{U \in \mathcal{U}_N^{(\beta)}} \langle e , X_N e\rangle ^{\frac{\beta N }{2}\theta_1} \Delta_{\tilde{\theta}}(Y_N)^{\frac{\beta N}{2}} dU \nonumber \\
 & =& \int_{e \in \Ss^{\beta N -1}} \langle e , X_N e\rangle ^{\frac{\beta N }{2}\theta_1} I_{N-1}\Big(\frac{N}{N-1}\tilde{\theta},Y_N\Big) de \label{cond}  .\end{eqnarray} 

Let us first determine the spectrum of $Y_N$ in function of $e$ and its components. Let us denote $P$ the characteristic polynomial of $X_N$. Using the linearity in the first column, we deduce: 

\[ P(z) = z \chi_{\tilde{X}_N} + \det\begin{pmatrix}  - a & - c^* \\  - c  & z - \tilde{X}_{N -1} \end{pmatrix} = z \chi_{\tilde{X}_N} - a \chi_{Y_N}(z) \]
where we used the Schur complement formula. 
In the same way we defined the $I^N_i$ and the $\gamma_i^N$ in subsection \ref{rank1dis}, we can also do it here: 

\[ \forall i= - m,...,k+l -m \, , I_i^N := \{ j \in [1,N], \lambda_{j}(X_N) = \mu_i \} \text{ and } \gamma_i^N = \sum_{j \in I_{i} } |e_j|^2 .\]
Then using Weyl interlacing relationships, we know that:

\[ \frac{\chi_{\tilde{X}_N}(z)}{P(z)} = \sum_{i=-m}^{l+k-m} \frac{\gamma^N_i}{z - \mu_i}
\]
therefore we have that: 

\[ \frac{\chi_{Y_N}(z)}{P(z)} =  \sum_{i=-m}^{l+k-m} \frac{\gamma^N_i \mu_i}{z - \mu_i}
.\]
Therefore, if we consider $Q$ the polynomial: 

\begin{equation} \label{roots}
	Q(z) = \sum_{i=-m}^{l+k-m} \gamma^N_i \mu_i \prod_{j \neq i} (z - \mu_i) 
 \end{equation}
and if we call $\chi_{-m} \leq ... \leq \chi_{-1} \leq \mu' _1\leq ... \mu'_l \leq \chi_1 \leq ... \leq \chi_{k-m-1}$. its roots, then the spectrum 
of $Y_N$ has the following form: 
\begin{equation}\label{spec1} Spec(Y_N) = \{ \chi_{-m},...,\chi_{-1},\underbrace{\mu_0,...,\mu_0}_{\alpha_0(N) -1}, \mu'_0 , \underbrace{\mu_1,...,\mu_1}_{\alpha_1(N) -1},....,\mu'_{l-1}, \underbrace{\mu_l,...,\mu_l}_{\alpha_l(N) -1}, \mu'_{l}, \chi_{1}, ...,\chi_{k-m-1} \} \end{equation} 

and the roots of $Q$ satisfy:

\begin{multline}\label{spec2} \mu_{-m} \leq \chi_{-m} \leq \mu_{-m+1} \leq ... \leq \mu_1 \leq \mu'_1 \leq ... \leq \mu_{l-1} \leq \mu'_{l-1}  \\ \leq \mu_l \leq \mu'_l \leq \mu_{l+1} \leq \chi_1 \leq \mu_{l+2} \leq ... \leq \chi_{k-m -1}\leq \mu_{k-m}. \end{multline}
In particular the spectrum of $Y_N$ is also interlaced with the spectrum of $X_N$. Therefore the empirical eigenvalue distribution of $Y_N$ also converges towards $\mu$. Furthermore, the $\chi_i$ and $\mu_i$ are continuous functions of the $\gamma^N_i$ (since they are roots of a polynomial whose coefficients are continuous in the $\gamma_i^N$) and if the $\gamma^N_i$ are all positive, they are respectively the $m$ smallest and $k -m -1$ largest solutions of the following equation in $z$.
\begin{equation}\label{interl} \sum_{i=-m}^{k+l-m} \frac{\gamma^N_i \mu_i}{z - \mu_i} = 0 \end{equation}

Given the form of the spectrum of $Y_N$ one can see $I_N(\tilde{\theta},Y_N)$ as a function of the $\mu'_i$ and $\chi_i$ and therefore as a function of of the $\gamma_i^N$.

 Then, using the recursion hypothesis, that is that Theorem \ref{maintheo} is satisfied for $k-1$, Remark \ref{equicont} that provides the equicontinuity of $I(\tilde{\theta},Y_N)$ in the $\mu'_i$ and $\chi_i$, and the continuity of the function $J$ in its first two arguments, one can prove the following lemma: 
 
 \begin{lemma}
 	Let $\mathcal{E}_N$ be the subset of matrices in $\mathcal{H}_{N-1}^{\beta, ++}$ whose spectrum is of the form described by the inequalities \eqref{spec1} with the $\mu'_i$ and $\chi_i$ satisfying the equation \eqref{spec2}. Then we have: 
 	\[ \lim_{N \to \infty} \sup_{Y \in \mathcal{E}_N} \Big| \frac{1}{N} \log I_{N-1}(\frac{N}{N-1} \tilde{\theta}, Y) - \sum_{i=1}^{m}J(\theta_{\sigma(i)}, \chi_ { i - m+1}, \mu) - \sum_{i= m+1}^{k-1} J(\theta_{\sigma(i)}, \chi_{i - m}, \mu) \Big| = 0 .\] 
 	\end{lemma}
If we look at the law of the $\gamma_i^N$, we have again that they follow a Dirichlet law with parameters: 

\[ (\underbrace{\frac{\beta}{2},...,\frac{\beta}{2}}_{ m \text { times }}, \frac{\beta \alpha_0(N)}{2},..., \frac{\beta \alpha_l(N)}{2}, \underbrace{\frac{\beta}{2},...,\frac{\beta}{2}}_{ k - m \text { times }}) \]
and therefore it follows a large deviation principle on the simplex $K$ with the rate function: 
\[ \gamma \mapsto \frac{\beta}{2}\sum_{i=0}^l \alpha_i \log \frac{\gamma_i}{\alpha_i} .\]
(where the indexation for the coordinates of $\gamma$ begins at $-m$ and ends at $k-m+l$ as for the $\gamma_i^N$). 
Therefore, using this large deviation principle in conjunction with the preceding lemma and with the equation \eqref{cond}, we can use Varadhan's lemma on $I_N(\bar{\theta}, X_N)$ so the limit of $N^{-1} \log I_N(\bar{\theta}, X_N)$ exists and is equal to the following supremum that we will denote $J(\bar{\theta},\{\mu_i\}, \{\alpha_i \})$: 

\[J(\bar{\theta},\{\mu_i\}, \{\alpha_i \}) = \frac{\beta}{2}\sup_{\gamma_i \geq 0, \atop \sum \gamma_i =1} \Big[\theta_1 \log \Big( \sum_{i= - m}^{l + k - m} \mu_i \gamma_i\Big) + \sum_{i=1}^{m}J(\theta_{\sigma(i)}, \chi_ { i - m+1}, \mu) + \sum_{i= m+1}^{k-1} J(\theta_{\sigma(i)}, \chi_{i - m}, \mu) + \sum_{i= 0}^{l} \alpha_i \ln \frac{\gamma_i}{\alpha_i} \Big]\]

where the $\chi_i$ are viewed as continuous functions of the $\gamma_i$ (that is as the $m$ smallest and $k-m-1$ largest roots of the polynomial $Q$ defined in equation \eqref{roots} where one replace $\gamma_i^N$ by $\gamma_i$). 
By continuity, we can take the supremum on the subset of $\gamma$ such that $\gamma_i >0$ for all $i$. Then since in this case, the $\chi_i$ are each the unique solution to equation \eqref{interl} on their corresponding intervals, we can instead take the preceding supremum over the couples $(\gamma, \chi)$ over the following domain:

\begin{multline*} \mathcal{D} = \Big\{ ( \gamma, \chi) \in (\R^{+,*})^{l +k +1} \times \prod_{i= -m}^{ -1} ] \mu_{ i} , \mu_{i+1}[  \times \prod_{i= 1}^{ k-m-1} ] \mu_{ l+i} , \mu_{l+i+1}[ : \\
	 \sum_{i=-m}^{k+l-m} \gamma_i =1, \forall j \in \llbracket -m, -1 \rrbracket \cup \llbracket 1, k-m-1 \rrbracket,  \sum_{i=-m}^{k+l-m} \frac{\gamma_i \mu_i}{\chi_j - \mu_i} = 0  \Big\} \end{multline*}
 And so we have :
\[J(\bar{\theta},\{\mu_i\}, \{\alpha_i \}) = \frac{\beta}{2}\sup_{(\gamma, \chi) \in \mathcal{D}} \Big[\theta_1 \log \Big( \sum_{i= - m}^{l + k - m} \mu_i \gamma_i\Big) + \sum_{i=1}^{m}J(\theta_{\sigma(i)}, \chi_ { i - m+1}, \mu) + \sum_{i= m+1}^{k-1} J(\theta_{\sigma(i)}, \chi_{i - m}, \mu) + \sum_{i= 0}^{l} \alpha_i \ln \frac{\gamma_i}{\alpha_i} \Big]\]

We are now going to make a change of variables and introduce the new $\overline{\gamma}_i$ in order to decouple the terms depending on the $\chi_i$ from the terms depending on the $\overline{\gamma}_i$. Namely we will let for $i \in \llbracket 0,l+1 \rrbracket$:

\[ \overline{\gamma}_i = \prod_{j= - m}^{-1} \frac{\chi_j(\mu_j - \mu_i)}{\mu_j(\chi_j - \mu_i)} \prod_{j= 1}^{k -m -1} \frac{\chi_j(\mu_{l+j+1} - \mu_i)}{\mu_{l+j+1}(\chi_j - \mu_i)} \gamma_i .\]

Let us prove that the fonction $\Phi : (\gamma,\chi) \mapsto (\overline{\gamma}, \chi)$ is a bijection of $\mathcal{D}$ onto the domain $\mathcal{D}'$ defined as: 

\[\mathcal{D}' = \Big\{ ( \overline{\gamma}, \chi) \in (\R^{+,*})^{l +2} \times \prod_{i= -m}^{ -1} ] \mu_{ i} , \mu_{i+1}[  \times \prod_{i= 1}^{ k-m-1} ] \mu_{ l+i} , \mu_{l+i+1}[ : \sum_{i=0}^{l+1} \overline{\gamma}_i =1  \Big\} .\]
First let us check that $\Phi(\mathcal{D}) \subset \mathcal{D}'$. It is easy to check that each term of the product that defines $\overline{\gamma}_i$ is positive so that $\overline{\gamma}_i >0$. Then let us define $F$ the following rational function:

\[ F(Y) = \prod_{j= - m}^{-1} \frac{\chi_j(Y \mu_j - 1)}{\mu_j(Y\chi_j - 1)} \prod_{j= 1}^{k -m -1} \frac{\chi_j(\mu_{l+j+1} Y - 1)}{\mu_{l+j+1}(Y\chi_j - 1)} \]

so that $\overline{\gamma}_i = F( \mu_i^{-1}) \gamma_i$. If we write down the partial fraction decomposition of $F$, we have that there exists a family of reals $a_j$ such that: 

\[ F(Y) = 1 + \sum_{j= -m}^{-1} \frac{a_j}{Y \chi_j -1} + \sum_{j= 1}^{k-m-1} \frac{a_j}{Y \chi_j -1} .\]

Furthermore, since $F(\mu_i^{-1}) =0$ for $i <0$ and $i > l+1$, we have that
\begin{eqnarray*} \sum_{i=0}^{l+1} \overline{\gamma}_i &=& \sum_{i=0}^{l=1} F(\mu_i^{-1}) \gamma_i \\ &=& \sum_{i=-m}^{l+k-m} F( \mu_i^{-1}) \gamma_i \\ &=& \sum_{i=-m}^{l+k-m} \gamma_i + \sum_{j=-m}^{-1}\sum_{i=-m}^{l+k-m}\frac{a_j\mu_i\gamma_i}{\chi_j - \mu_i} + \sum_{j=1}^{k-m-1}\sum_{i=-m}^{l+k-m}\frac{a_j\mu_i\gamma_i}{\chi_j - \mu_i} \\
	&=&1
 \end{eqnarray*} 
Where we used the interlacing equation satisfied by $\gamma$ and $\chi$ in the last line. Furthermore, $F(\mu_i^{-1}) >0$ so $\overline{\gamma}_i > 0$. 
To prove the bijection, one can refer to \cite{GuiHus} where a similar change of variable is made in the proof of proposition 14. In particular we have that for $i \in\llbracket -m, l+k-m \rrbracket \setminus \llbracket 0, l+1 \rrbracket$, $\gamma_i = - (\sum_{j=0}^{l+2} G_i(\mu_j^{-1}) \overline{\gamma_j})/ G_i( \mu_{j}^{-1}) $ where $G_i(X) = F(X)(X \mu_i - 1)^{-1}$. 

%
%
%
%
%
%

Let us now compare $\sum \mu_i \gamma_i$ with $\sum \mu_i \overline{\gamma}_i$, for this, let us write down the partial fraction decomposition for $F(Y)/Y$: 

\[ \frac{F(Y)}{Y} = \frac{b}{Y} +  \sum_{j= -m}^{-1} \frac{b_j}{Y \chi_j -1} + \sum_{j= 1}^{k-m-1} \frac{b_j}{Y \chi_j -1} \]

where 

\[ b = \prod_{j= - m}^{-1} \frac{\chi_j}{\mu_j} \prod_{j= 1}^{k - m -1} \frac{\chi_j}{\mu_{l+j+1}} \]
so that, substituting $Y$ for $\mu_{i}^{-1 }$, multiplying by $\gamma_i$ and summing, one obtains:

\[ \sum_{i=0}^{l+1} \mu_i \overline{\gamma_i} = b \sum_{i=-m}^{l+k-m} \mu_i \gamma_i. \] 

Therefore, we can write that 

\begin{eqnarray*} J(\bar{\theta},\{\mu_i\}, \{\alpha_i \}) &=& \frac{\beta}{2}\sup_{(\overline{\gamma}, \chi) \in \mathcal{D}'} \Big[ \theta_1 \ln \Big( \sum_{i= 0}^{l + 1} \mu_i \overline{\gamma}_i\Big) + \sum_{i= 0}^{l} \alpha_i \ln \frac{\overline{\gamma}_i}{\alpha_i} \\ 
	&+ &\sum_{i=1}^{m} \Big( J(\theta_{\sigma(i)}, \chi_ { i - m+1}, \mu) + (\theta_1 +1)\ln \mu_{i - m+1}  - (\theta_1 +1) \ln \chi_{i - m+1} \\ &+& \sum_{j=0}^l \alpha_j \ln( \mu_j - \chi_{i-m+1}) -\sum_{j=0}^l \alpha_j \ln( \mu_j - \mu_{i-m+1}) \Big)
	 \\ &+&\sum_{i= m+1}^{k-1} \Big( J(\theta_{\sigma(i)}, \chi_{i -m}, \mu) + (\theta_1 +1) \ln \mu_{l + i -m+1} - (\theta_1 +1)\ln \chi_{i- m}  \\
	 &+& \sum_{j=0}^l \alpha_j \ln(  \chi_{i-m} - \mu_j) -\sum_{j=0}^l \alpha_j \ln( \mu_{l+ i-m+1} - \mu_j ) \Big) \Big].
\end{eqnarray*}

We define the quantities $H_i$ for $i = 1,...,k-1$ as follows : 

\[ H_i(\chi)  = J(\theta_{\sigma(i)}, \chi, \mu) - (\theta_1 + 1) \ln \chi + \sum_{j=0}^l \alpha_j \ln( |\mu_j - \chi|)  \]

so that: 

	\begin{eqnarray*} J(\bar{\theta},\{\mu_i\}, \{\alpha_i \}) &=& \frac{\beta}{2}\sup_{(\overline{\gamma}, \chi) \in \mathcal{D}'} \Big[ \theta_1 \ln \Big( \sum_{i= 0}^{l + 1} \mu_i \overline{\gamma}_i\Big) + \sum_{i= 0}^{l} \alpha_i \ln \frac{\overline{\gamma}_i}{\alpha_i} \\ 
		&+ &\sum_{i=1}^{m} \Big( H_i(\chi_{i-m+1})+ (\theta_1 +1)\ln \mu_{i - m+1}   -\sum_{j=0}^l \alpha_j \ln( \mu_j - \mu_{i-m+1}) \Big)
		\\ &+&\sum_{i= m+1}^{k-1} \Big( H_i(\chi_{i-m})+ (\theta_1 +1) \ln \mu_{l + i -m}  -\sum_{j=0}^l \alpha_j \ln( \mu_{l+ i-m} - \mu_j ) \Big) \Big].
	\end{eqnarray*}

Using the fact that the domain $\mathcal{D}'$ is a product, this can be further simplified in:

\begin{eqnarray*} J(\bar{\theta},\{\mu_i\}, \{\alpha_i \}) &=& \frac{\beta}{2}\sup_{\overline{\gamma}_i \geq 0 \atop \sum_i \overline{\gamma}_i = 1} \Big[ \theta_1 \ln \Big( \sum_{i= 0}^{l + 1} \mu_i \overline{\gamma}_i\Big) + \sum_{i= 0}^{l} \alpha_i \ln \frac{\overline{\gamma}_i}{\alpha_i} \Big]\\ 
	&+ &\frac{\beta}{2}\sum_{i=1}^{m}\Big( \sup_{\chi \in ] \mu_{i-m+1}, \mu_{i-m+2} [ }  H_i(\chi)+ (\theta_1 +1)\ln \mu_{i - m+1}   -\sum_{j=0}^l \alpha_j \ln( \mu_j - \mu_{i-m+1}) \Big)
	\\ &+&\frac{\beta}{2}\sum_{i= m+1}^{k-1} \Big( \sup_{\chi \in ] \mu_{l+i-m}, \mu_{l+i-m+1} [ } H_i(\chi)+ (\theta_1 +1) \ln \mu_{l + i -m+1}  -\sum_{j=0}^l \alpha_j \ln( \mu_{l+ i-m+1} - \mu_j ) \Big). 
\end{eqnarray*}

Here, following the analysis of the same optimization problem performed in subsection \ref{rank1dis}, the first line is going to be equal to $J(\theta_1, \mu_{l+1}, \mu)$. Let us determine the variations of $H_i$ on the different intervals considered. 	
For $i =1,...,m$
we have that $H_i$ is decreasing and therefore its supremum on $] \mu_{i -m -1}, \mu_{i-m} [$ is $H_i(\theta_1, \mu_{i-m-1})$ and the term of index $i$ simplifies in $J(\theta_{\sigma(i)}, \mu_{i-m-1}, \mu)$.

For $i = m+1,..., k-1$, if $\chi \leq T_{\mu}^{-1}( \theta_{\sigma(i)})$, $\chi \mapsto J(\theta_{\sigma(i)}, \chi,\mu)$ is constant and therefore:

\[ H_i'(\chi) = - \frac{\theta_1 +1}{\chi} + \sum_{j=0}^l \alpha_j \frac{1}{ \chi - \mu_j} = G_{\mu}( \chi) - \frac{\theta_1 +1}{\chi}.\]

This derivative is positive or null if $\chi \leq T_{\mu}^{-1}(\theta_1)$ and negative or null if $\chi \geq T_{\mu}^{-1}(\theta_1)$. Now, if $\chi \geq T_{\mu}^{-1}( \theta_{\sigma(i)})$, we have 

\[ H_i'(\chi) = \frac{\theta_{\sigma(i)} -\theta_1}{\chi}. \]

Let us look first at the terms of indices $i \geq p+1$, that is such that $\theta_{\sigma(i)} \geq \theta_1$. Then, since $T_{\mu}^{-1}(\theta_{\sigma(i)}) \leq T_{\mu}^{-1}(\theta_1) $, $H_i$ is increasing and therefore, on any bounded intervals attains its supremum for values that tends to its rightmost point, and we have:

\[ \sup_{\chi \in ] \mu_{l+i-m}, \mu_{l+i-m+1} [ } H_i(\chi) = H(\mu_{l+i-m+1}). \]
Using the expression of $H_i$, we have that the term of index $i$ in the preceding sum simplifies in $J(\theta_{\sigma(i)}, \mu_{l+i -m+1})$. 
 We end up with the simplified expression: 
 
 \begin{eqnarray*}J(\bar{\theta},\{\mu_i\}, \{\alpha_i \}) &=& \frac{\beta}{2}J(\theta_1, \mu_{l+1}, \mu)\\ 
 &+ &\frac{\beta}{2}\sum_{i=1}^{m} J(\theta_{\sigma(i)}, \mu_{i-m-1}, \mu)
 \\ &+&\frac{\beta}{2}\sum_{i= m+1}^{p } \Big( \sup_{\chi \in ] \mu_{l+i-m}, \mu_{l+i-m+1} [ } H_i(\chi)+ (\theta_1 +1) \ln \mu_{l + i -m+1}  -\sum_{j=0}^l \alpha_j \ln( \mu_{l+ i-m+1} - \mu_j ) \Big) \\
 &+& \frac{\beta}{2}\sum_{i=p+1}^{k-1} J(\theta_{\sigma(i)}, \mu_{i - m +l+1}, \mu).
 \end{eqnarray*}
 
Now let us look at the variation of $H_i$ for $m+1 \leq i \leq p$. Since $0 \leq \theta_{\sigma(i)} \leq \theta_1$, $T_{\mu}^{-1}( \theta_1) \leq T_{\mu}^{-1}(\theta_{\sigma(i)})$ and $H_i$ is increasing till $T_{\mu}^{-1}(\theta_1)$ and decreasing after. We let

\[ q = \sup \{ i \in \llbracket m+1,p+1 \rrbracket , T_{\mu}( \mu_{l+i-m}) \geq \theta_{1}\} \]

We will assume here that $q \neq - \infty , p+1$, we leave those two edge cases to the reader since they are in fact easier to deal with.

For every $i < q$ we have that:
\[ \sup_{\chi \in ] \mu_{l+i-m}, \mu_{l+i-m+1} [ } H_i(\chi) = H_i( \mu_{l+i-m+1}) = J(\theta_{\sigma(i)}, \mu_{l+i-m+1},\mu) - (\theta_1 +1) \ln \mu_{l+i -m +1} + \sum_{j=0}^l \alpha_j \ln( \mu_{l+i-m+1} - \mu_j)\]

Furthermore since $\theta_{\sigma(i)} \leq \theta_1 \leq  T_{\mu}^{-1}(\mu_{l+i-m+1})$, we have $J(\theta_{\sigma(i)}, \mu_{l+i-m+1},\mu) = J(\theta_{\sigma(i)}, \mu_{l+i-m}, \mu)$ and therefore the term of index $i$ in the sum above is $J(\theta_{\sigma(i)}, \mu_{l+i-m}, \mu)$.

For $i = q$, we have that

\[ \sup_{\chi \in ] \mu_{l+i-m}, \mu_{l+i-m+1} [ } H_i(\chi) = H_i( T_{\mu}^{-1}( \theta_1)) \]

and therefore the term of index $ i=q$ is: 

\[ J(\theta_{\sigma(q)},T_{\mu}^{-1}(\theta_1), \mu) + (\theta_1+1)( \ln \mu_{l+q-m+1} - \ln T_{\mu}^{-1}(\theta_1)) - \sum_{j=0}^l \alpha_j \ln( \mu_{l+q-m+1} - \mu_j) + \sum_{j=0}^l \alpha_j \ln( T_{\mu}^{-1}( \theta_1) - \mu_j) \]

Furthermore, looking back at the expression of $J$, we have that the expression above is: 
\[ J(\theta_{\sigma(q)},\mu_{l+q-m}, \mu) + J(\theta_1, \mu_{l+q-m+1},\mu) -J(\theta_1, T^{-1}_{\mu}(\theta_1),\mu) . \]

For every $ i > q$, we have that 
 
\[ \sup_{\chi \in ] \mu_{l+i-m}, \mu_{l+i-m+1} [ } H_i(\chi) = H_i( \mu_{l+i-m}). \]

And therefore, the term of index $i$ is: 
\[ J(\theta_{\sigma(i)},\mu_{l+i-m}, \mu) + (\theta_1+1)( \ln \mu_{l+i-m+1} - \ln \mu_{l+i-m}) - \sum_{j=0}^l \alpha_j \ln( \mu_{l+i-m+1} - \mu_j) + \sum_{j=0}^l \alpha_j \ln( \mu_{l+i-m} - \mu_j). \]

There again, this term is equal to: 
\[ J(\theta_{\sigma(i)},\mu_{l+i-m}, \mu) + J(\theta_1, \mu_{l+i-m+1},\mu) -J(\theta_1, \mu_{l+i-m},\mu) .\]

Rearranging the terms in $J$, we have 

\begin{eqnarray*}
	J(\bar{\theta},\{\mu_i\}, \{\alpha_i \}) &=& \frac{\beta}{2}J(\theta_1, \mu_{l+1}, \mu)\\ 
	&+ &\frac{\beta}{2}\sum_{i=1}^{m} J(\theta_{\sigma(i)}, \mu_{i-m-1}, \mu)
	\\ &+&\frac{\beta}{2}\sum_{i= m+1}^{q-1 } J(\theta_{\sigma(i)},\mu_{l+i-m},\mu) \\
		&+& J(\theta_{\sigma(q)},T_{\mu}^{-1}(\theta_1), \mu) + J(\theta_1,\mu_{l+q-m+1}, \mu) - J(\theta_1,T_{\mu}^{-1}(\theta_1), \mu)
		\\
	&+&
	\frac{\beta}{2}\sum_{i= q+1}^{p }J(\theta_{\sigma(i)},\mu_{l+i-m}, \mu) + J(\theta_1,\mu_{l+i-m+1}, \mu) - J(\theta_1,\mu_{l+i-m}, \mu)\\
	&+& \frac{\beta}{2} \sum_{i=p+1}^{k-1} J(\theta_{\sigma(i)}, \mu_{i - m +l+1}, \mu)
\end{eqnarray*}

Since $\mu_{l+1} \leq T_{\mu}^{-1}(\theta_1)$ (we assumed that $q > m+1$), $J(\theta_1, \mu_{l+1}, \mu)=J(\theta_1, T_{\mu}^{-1}(\theta_1), \mu)$ and the sum above simplifies in 
\begin{eqnarray*}
	J(\bar{\theta},\{\mu_i\}, \{\alpha_i \}) &=& \frac{\beta}{2}J(\theta_1, \mu_{l+ p-m+1}, \mu)\\ 
	&+ &\frac{\beta}{2}\sum_{i=1}^{m} J(\theta_{\sigma(i)}, \mu_{i-m-1}, \mu)
	\\ &+&\frac{\beta}{2}\sum_{i= m+1}^{q-1 } J(\theta_{\sigma(i)},\mu_{l+i-m},\mu) \\
	&+& J(\theta_{\sigma(q)},\mu_{l+q-m}, \mu) \\
	&+&\frac{\beta}{2}\sum_{i= q+1}^{p }J(\theta_{\sigma(i)},\mu_{l+i-m}, \mu)\\
	&+& \frac{\beta}{2} \sum_{i=p+1}^{k-1} J(\theta_{\sigma(i)}, \mu_{i - m +l+1}, \mu)
\end{eqnarray*}
which is the expected limit and therefore the lemma is proved. 

\end{proof}
For $\theta_1 \leq 0$, the computation remain the same except that the role of the  indices $i$ between $1$ and $m$ and the role of the  indices $i$ between $m+1$ and $k$ are switched.
\subsection{Case of a general $\mu$}\label{rank2cont}

To generalize our result, one first need a multi-dimensional version of the continuity lemma \ref{cont}:
\begin{lemma}
	Let $\kappa > 0$. The function 

\[ (\theta,\lambda,\mu) \mapsto \sum_{i=1}^k J(\theta_i, \lambda_i, \mu) \]
defined on the following domain: 
\[ \mathcal{D}^{m, k-m} := \{(\theta, \lambda, \mu) \in  (\R^{-})^m \times (\R^{+})^{k-m} \times (\R^{+,*})^k \times \mathcal{P}(\R^+) : \lambda_1 \leq ... \leq \lambda_{m} \leq l( \mu) \leq r(\mu) \leq \lambda_{m+1} \leq ... \leq \lambda_k \leq \kappa \} \]

endowed with the topology induced by the product topology, is continuous. 
\end{lemma}
This lemma can be easily be proved from Lemma \ref{cont}.  Then, using again continuity arguments, we can see that the result of the previous subsection still stand for sequences of matrices $(X_N)_{N \in \N}$ of the same form where we relaxed the conditions on the $\mu_i$ to only $\mu_{-m} \leq  \mu_{- m+1} \leq ... \leq \mu_{l + k-m}$. Then for general eigenvalue distribution $\mu$ we proceed the same way as in subsection \ref{rank1cont}, by approximating uniformly the eigenvalues.

\section{Conclusion and open questions}
In this note, we proved one of the main results of \cite{MerPot} in the case $\beta =1,2$ as well as the conjecture given as the last Remark of section 2 of \cite{MerPot}. An interesting thing that this result teaches us is that, although the lack of symmetry of the integral in the $\theta_i$ disappears asymptotically for parameters $\bar{\theta}$ whose support lie on the $k$ first indices, it is no more the case if the support is not restricted to $\llbracket 1,k \rrbracket$, even if it remains finite. To see this, one can take some sequence $X_N$ satisfying the hypothesis of Theorem \ref{maintheo} whose largest eigenvalue converges toward $\lambda$ and whose eigenvalue distribution converges toward $\mu$, and let $\theta = (1,0,...,0)$ and $\theta' = (0,....,0,1)$. then, one can write that 
\[ \Delta_{\frac{\beta N}{2}\theta'}(U^* X_N U) =\Big( \frac{\det(U^* X_N U)}{\det( [U^* X_N U] _{N -1})} \Big)^{\frac{\beta N}{2}} = (U^* X^{-1}_N U)_{1,1}^{- \frac{\beta N}{2}} \]
and therefore: 
\[ I_N(\theta',X_N) = I_N(-\theta,X_N^{-1}) \]

Thus 

\[\lim_{N \to \infty}  N^{-1} \log I_N( \theta',X_N) = \frac{\beta}{2}J \Big( -1,  \lambda^{-1}, \frac{1}{\mu} \Big) \]
where $1/\mu$ is the push-forward of $\mu$ by $x \mapsto 1/x$. But, since $-1 \geq T_{1/\mu}( 1/ \lambda)$, one has: 

\[ J \Big( -1,  \lambda^{-1}, \frac{1}{\mu} \Big)= - \int \log |x| d \frac{1}{\mu}(x) = \int \log x d \mu(x) \]
 This quantity is in general different from $J(1, \lambda, \mu)$ (since in particular, it never depends on $\lambda$ whatever its value or the value of $\mu$). An interesting result would be to study given $\theta \in \R$ and $\theta^{(M)} = ( 0,...,0, \underbrace{\theta}_{\text{position } M},....0)$, the behavior of $N^{-1} \log I_N( \theta^{(M(N))}, X_N)$ depending of the behavior of the sequence $M(N)$.

Another natural question already asked by Mergny and Potters is the question of the behavior of the integral for arguments $\bar{\theta}$ whose support tends to $\infty$. For instance, if $\bar{\theta}^N = (\theta_1^N,...,\theta_{k(N)}^N)$ with $k(N)$ some sequence whose limit is $+ \infty$, can we prove the convergence of $(k(N) N) ^{-1} \log I_N(\bar{\theta}^N, X_N)$ toward some limit and then does the lack of symmetry in $\theta$ disappears at $\infty$?

\bibliographystyle{plain}

\bibliography{biblioStransform}

\begin{thebibliography}{10}

\bibitem{Giulio}
Giulio Biroli and Alice Guionnet.
\newblock Large deviations of spiked random matrices.
\newblock {\em To appear in ECP}, 2020.

\bibitem{cuenca2017pieri}
Cesar Cuenca.
\newblock Pieri integral formula and asymptotics of jack unitary characters,
  2017.

\bibitem{2018}
Cesar Cuenca.
\newblock Asymptotic formulas for macdonald polynomials and the boundary of the
  (q,t)-gelfand-tsetlin graph.
\newblock {\em Symmetry, Integrability and Geometry: Methods and Applications},
  Jan 2018.

\bibitem{Eyn}
Bertrand Eynard.
\newblock Eigenvalue distribution of large random matrices, from one matrix to
  several coupled matrices.
\newblock {\em Nuclear Physics B}, 506(3):633–664, Dec 1997.

\bibitem{2015}
Vadim Gorin and Greta Panova.
\newblock Asymptotics of symmetric polynomials with applications to statistical
  mechanics and representation theory.
\newblock {\em The Annals of Probability}, 43(6), Nov 2015.

\bibitem{HCIZalg}
Ian~P. Goulden, Mathieu Guay-Paquet, and Jonathan Novak.
\newblock Monotone hurwitz numbers and the hciz integral.
\newblock {\em Annales Math\'ematiques Blaise Pascal}, 21(1):71--89, 2014.

\bibitem{HuGu}
Alice Guionnet and Jonathan Husson.
\newblock Large deviations for the largest eigenvalue of {R}ademacher matrices.
\newblock {\em Ann. Probab.}, 48(3):1436--1465, 2020.

\bibitem{GuiHus}
Alice Guionnet and Jonathan Husson.
\newblock Asymptotics of k dimensional spherical integrals and applications.
\newblock {\em arXiv:2101.01983}, 2021.

\bibitem{GuiMai}
Alice Guionnet and Myl\`ene Ma{\i}da.
\newblock A fourier view on the r-transform and related asymptotics of
  spherical integrals.
\newblock {\em Journal of functional analysis}, 222(2):435--490, 2005.

\bibitem{Har}
Harish-Chandra.
\newblock Invariant differential operators on a semisimple {L}ie algebra.
\newblock {\em Proc. Nat. Acad. Sci. U.S.A.}, 42:252--253, 1956.

\bibitem{HeOp}
G.~J. Heckman and E.~M. Opdam.
\newblock Root systems and hypergeometric functions. {I}.
\newblock {\em Compositio Mathematica}, 64(3):329--352, 1987.

\bibitem{Hu}
Jonathan Husson.
\newblock Large deviations for the largest eigenvalue of matrices with variance
  profiles.
\newblock {\em arXiv:2002.01010}, 2020.

\bibitem{ItzZub}
Claude Itzykson and Jean-Bernard Zuber.
\newblock The planar approximation. {II}.
\newblock {\em J. Math. Phys.}, 21:411--421, 1980.

\bibitem{MerPot}
Pierre Mergny and Marc Potters.
\newblock Asymptotic behavior of the multiplicative counterpart of the
  harish-chandra integral and the $s$-transform.
\newblock {\em arXiv:2007.09421}, 2021.

\bibitem{HCIZalg2}
Jonathan Novak.
\newblock On the complex asymptotics of the hciz and bgw integrals.
\newblock {\em arXiv:2006.04304}, 2020.

\bibitem{OpBook}
Eric~M. {Opdam}.
\newblock {\em {Lecture notes on Dunkl operators for real and complex
  reflection groups}}, volume~8.
\newblock Tokyo: Mathematical Society of Japan, 2000.

\end{thebibliography}
\end{document}